\documentclass{amsart}

\usepackage[active]{srcltx}

\usepackage{graphicx}

\usepackage{latexsym}
\usepackage{amsmath,amsthm}
\usepackage{amsfonts}
\usepackage[psamsfonts]{amssymb}
\usepackage{enumerate}

\usepackage[final]
{showkeys}
\usepackage{url}

\RequirePackage[OT1]{fontenc}

\usepackage{calrsfs}

\usepackage{tikz}
\usetikzlibrary{decorations}
\usetikzlibrary{decorations.pathreplacing}

\usepackage{hyperref}

\theoremstyle{plain} 
\newtheorem{theorem}{Theorem}[section]

\newtheorem{proposition}[theorem]{Proposition}

\newtheorem*{conjecture*}{Conjecture}
\theoremstyle{definition} 

\theoremstyle{definition} 

\newtheorem{remark}[theorem]{Remark}
\newtheorem*{remark*}{Remark}
\numberwithin{equation}{section}

\makeatletter
  \renewcommand\section{\@startsection {section}{1}{\z@}%
                                   {-\bigskipamount}%
                                   {\medskipamount}%
                                   {\large\bfseries
                                   \raggedright}}

  \renewcommand\subsection{\@startsection {subsection}{2}{\z@}%
                                   {-\medskipamount}%
                                   {\smallskipamount}%
                                   {\bfseries
                                   \raggedright}}
\makeatother                                   

\renewcommand{\c}{\mathsf{c}}
\newcommand{\fin}{\mathsf{fin}}
\newcommand{\Ca}{\mathsf{Ca}}
\newcommand{\Co}{\mathsf{Co}}

\renewcommand{\gg}{>\kern-2pt>}
\renewcommand{\ll}{<\kern-2pt<}

\renewcommand{\gg}{>\kern-2pt>}
\renewcommand{\ll}{<\kern-2pt<}

\renewcommand{\le}{\leqslant}
\renewcommand{\ge}{\geqslant}

\newcommand{\al}{\alpha}

\newcommand{\si}{\sigma}

\newcommand{\Om}{
X}

\newcommand{\M}{\mathcal{M}}

\newcommand{\LD}{\mathcal{L}\!\mathcal{D}}
\renewcommand{\LD}{\mathcal{L}{\kern -1.9pt}\mathcal{D}}
\renewcommand{\LD}{\mathcal{D}}
\renewcommand{\LD}{\mathcal{L}{\kern -4pt}\mathcal{C}}
\renewcommand{\LD}{\mathcal{R}{\kern -3pt}\mathcal{C}}

\newcommand{\Z}{\mathbb{Z}}
\newcommand{\R}{{\mathbb{R}}}

\newcommand{\tm}{{\tilde{m}}}

\newcommand{\A}{\mathcal{A}}

\renewcommand{\d}{\mathrm{d}}

\newcommand{\vp}{\varepsilon}

\begin{document}


\title{Measure extension by local approximation}


\author{Iosif Pinelis}

\address{Department of Mathematical Sciences\\
Michigan Technological University\\
Hough\-ton, Michigan 49931, USA\\
E-mail: ipinelis@mtu.edu}



%

\begin{abstract}
Measurable sets are defined as those locally approximable, in a certain sense, by sets in the given algebra (or ring). A corresponding measure extension theorem is proved. 
It is also shown that a set is locally approximable in the mentioned sense if and only if it is Carath\'eodory-measurable.   
\end{abstract}

\maketitle

\tableofcontents


\section{Introduction, 
summary, 
and discussion
}\label{intro} 

Let $m\colon\A\to[0,\infty]$ be a measure -- that is, a nonnegative $\si$-additive function defined on an algebra $\A$ over a set $\Om$ such that $m(\emptyset)=0$. 

The measure extension problem is to extend the measure $m$ to a measure on a $\si$-algebra containing $\A$. This problem was solved by Carath\'eodory; see e.g.\ \cite{halmos}. The key in his solution was to consider the set 
\begin{equation}\label{eq:M_Ca}
	\M_\Ca:=\{E\subseteq\Om\colon m^*(F)=m^*(F\cap E)+m^*(F\cap E^\c)\ \; \forall F\subseteq\Om\}  
\end{equation}
of all Carath\'eodory-measurable subsets of $\Om$, where, as usual, $m^*$ denotes the outer measure corresponding to $m$, and ${}^\c$ denotes the complement (to $\Om$). 
It is then shown that $\M_\Ca$ is a $\si$-algebra containing $\A$, and the restriction of $m^*$ to $\M_\Ca$ is a measure extending $m$. 

When the measure $m$ is finite, one can also introduce the inner measure $m_*$ by the formula $m_*(E):=m(\Om)-m^*(E^\c)$ for all $E\subseteq\Om$, and then the key condition $m^*(F)=m^*(F\cap E)+m^*(F\cap E^\c)$ in \eqref{eq:M_Ca} can be rewritten in the case $F=\Om$ as $m^*(E)=m_*(E)$. This equality of the outer and inner measures on all Carath\'eodory-measurable subsets of $\Om$ may explain the intuition behind the definition \eqref{eq:M_Ca}. 

Moreover, one can show -- see e.g.\ Theorem~\ref{th:M comp} at the end of this section -- that the condition $\forall F\subseteq\Om$ in \eqref{eq:M_Ca} can be equivalently replaced by $\forall A\in\A_\fin$, where 
\begin{equation}\label{eq:A_fin}
	\A_\fin:=\{A\in\A\colon m(A)<\infty\}.   
\end{equation}
That is, 
\begin{equation}\label{eq:M_Ca,loc}
	\M_\Ca=\bigcap_{A\in\A_\fin}\M_A,
\end{equation}
where 
\begin{equation*}
	\M_A:=\{E\subseteq\Om\colon(m_A)^*(A\cap E)=(m_A)_*(A\cap E)\}, 
\end{equation*}
$m_A$ is the restriction of the measure $m$ to the algebra $\A_A:=\{B\in\A\colon B\subseteq A\}$ over the set $A$, and $(m_A)^*$ and $(m_A)_*$ are the outer and inner measures corresponding to $m_A$. 
This ``localized'' restatement of the definition of $\M_\Ca$ brings it closer to the mentioned intuition of the desired equality of the outer and inner measures of measurable sets. 
  
Another approach to the measure extension problem 
is based on an approximation idea, which may be more immediately intuitive. 
For any subsets $E$ and $F$ of $\Om$, define the ``distance'' between them by the formula  
\begin{equation}\label{eq:d}
	d(E,F):=m^*(E+F), 
\end{equation}
where $E+F$ denotes the symmetric difference between $E$ and $F$. 
The idea is then to define the set of all measurable subsets of $\Om$ as the closure of the algebra $\A$ with respect to the pseudometric $d$. 
This idea was carried out in \cite[Appendix~1]{borovkov} in the case when $m$ is a probability measure. Of course, one can quite similarly do for any finite measure $m$; cf.\ e.g.\ \cite[Theorem~1.5.6]{bogachev}. 

However, without modifications, this approach will not work in general even when the measure $m$ is $\si$-finite. For instance, suppose that $\Om=\R$, $\A$ is the smallest algebra containing all left-open intervals $(a,b]$ in $\R$, and $m$ is the Lebesgue measure on $\A$, so that $m$ is $\si$-finite. Let now $E:=\bigcup_{n\in\Z}(2n,2n+1]$, so that $E$ is in the $\si$-algebra $\si(\A)$ generated by $\A$. Then it is easy to see that $d(E,A)=\infty$ for any $A\in\A$. Moreover, \cite[Example~4.19]{wise-hall} shows that there exist a set $\Om$, an algebra $\A$ over $\Om$, a $\si$-finite measure $\mu$ on $\si(\A)$, and a set $E\in\si(\A)$ such that $\mu(E)<\infty$ but $\mu(E+A)\ge\mu(E)>0$ for all $A\in\A$. 

The approximation idea can be saved, though, by combining it with appropriate localization. 
That is, a measurable set may be only ``locally'' approximable by sets in the algebra $\A$.  
Specifically, for any $A\in\A$ consider the following ``localized'' version of the definition \eqref{eq:d}: 
\begin{equation*}
	d_A(E,F):=m^*\big(A\cap(E+F)\big) 
\end{equation*}
for any subsets $E$ and $F$ of $\Om$. 
Thus, for the ``distance'' $d$ defined by \eqref{eq:d}, we have $d=d_\Om$.

Now recall \eqref{eq:A_fin} and let $\M$ denote the set of all subsets $E$ of $\Om$ such that for each $A\in\A_\fin$ and each real $\vp>0$ there is some $B=B_{E,A,\vp}\in\A$ such that $d_A(E,B)<\vp$:
\begin{equation}\label{eq:M}
	\M:=\{E\subseteq\Om\colon\forall A\in\A_\fin\;\forall\vp>0\;\exists B\in\A\ d_A(E,B)<\vp\}. 
\end{equation}
Note that here we use the ``$m$-finite'' subset $\A_\fin$ of the algebra $\A$ (rather than $\A$ itself); this localization idea is similar to the one that led us to \eqref{eq:M_Ca,loc}. 

\begin{theorem}\label{th:si-alg}
$\M$ is a $\si$-algebra over $\Om$, and $\M\supseteq\A$. 
\end{theorem}

The necessary proofs are given in Section~\ref{proofs}.


\begin{theorem}\label{th:si-add}
The outer measure $m^*$ is $\si$-additive on the 
$\si$-algebra $\M$. 
\end{theorem}

So, in view of \eqref{eq:m*=m}, the restriction 
\begin{equation*}
	\bar m:=m^*\big|_\M
\end{equation*}
of $m^*$ to $\M$ is a measure that extends $m$ from the algebra $\A$ to the $\si$-algebra $\M$.

\begin{theorem}\label{th:uniq}
If the measure $m$ on $\A$ is $\si$-finite, then the $\si$-additive extension of $m$ to $\M$ is unique. 
\end{theorem}

Of course, the $\si$-finiteness condition in Theorem~\ref{th:uniq} is essential; for instance, see \cite[Example~4.20]{wise-hall}.   

Let us also compare the set $\M$, defined by \eqref{eq:M}, of all sets locally approximable by sets in algebra $\A$ with the set $\M_\Ca$, defined by \eqref{eq:M_Ca}, 
of all sets measurable in the Carath\'eodory sense, as well as with the completion 
\begin{equation*}
	\M_\Co:=\{E\subseteq\Om\colon \exists S\in\si(\A)\ d_\Om(E,S)=0\}
\end{equation*}
of the $\si$-algebra $\si(\A)$ generated by algebra $\A$. 
Let us also consider the following ``$\A_\fin$'' counterparts of the Carath\'eodory set $\M_\Ca$ and the set $\M_\Co$: 
\begin{align*}
	\M_{\Ca,\A_\fin}&:=\{E\subseteq\Om\colon m^*(A)=m^*(A\cap E)+m^*(A\cap E^\c)\ \;\forall A\in\A_\fin\}, \\ 
	\M_{\Co,\A_\fin}&:=\{E\subseteq\Om\colon\forall A\in\A_\fin\ \exists S\in\si(\A)\ d_A(E,S)=0\}.  
\end{align*}

\begin{theorem}\label{th:M comp}
$\M_\Co\subseteq	\M_{\Co,\A_\fin}=\M=\M_\Ca=\M_{\Ca,\A_\fin}$. 
If the measure $m$ on $\A$ is $\si$-finite, then $\M_\Co=	\M_{\Co,\A_\fin}=\M=\M_\Ca=\M_{\Ca,\A_\fin}$.  
\end{theorem}
In \cite[Theorem~2.3]{weizs-meas}, it was shown that the restriction of $m^*$ to $\M_\Co$ is $\si$-additive. 
%
%
The $\si$-finiteness condition in the second sentence of Theorem~\ref{th:M comp} is essential; cf.\ e.g.\  \cite[Example~4.28]{wise-hall}. 

\begin{remark}
The condition $\Om\in\A$ will never be used in the proofs of Theorems~\ref{th:si-alg}--\ref{th:M comp} (to be given in Section~\ref{proofs}). 
So, these theorems 
will hold even if $\A$ is only assumed to be a ring (but not necessarily an algebra) of subsets of $\Om$. 
\end{remark}

\section{Proofs}\label{proofs}
First here, let us recall the definition and basic properties of the outer measure corresponding to the given measure $m$ on the algebra $\A$. 

Take any set $E\subseteq\Om$. Let $c(E)$ denote the set of all sequences $(A_n):=(A_n)_{n=1}^\infty$ in $\A$ such that $\bigcup_n A_n\supseteq E$. Let also $c_\d(E)$ denote the set of all disjoint sequences $(A_n)\in c(E)$, so that $A_i\cap A_j=\emptyset$ whenever $(A_n)\in c_\d(E)$ and $i\ne j$. Consider the outer measure 
\begin{equation}\label{eq:outer}
m^*(E):=\inf\Big\{\sum_n m(A_n)\colon(A_n)\in c(E)\Big\}	
\end{equation}
of the set $E$ corresponding to the measure $m$ on the algebra $\A$. 
The following properties of the outer measure are well known and easy to check: 
for any subsets $E,E_1,E_2,\dots$ of $\Om$, one has 
\begin{description}
	\item[{\bf positivity}] $m^*(E)\ge0$; 
	\item[{\bf monotonicity}] if $E_1\subseteq E_2$ then $m^*(E_1)\le m^*(E_2)$; 
	\item[{\bf subadditivity}] $m^*\big(\bigcup_n E_n\big)\le\sum_n m^*(E_n)$; 
	\item[{\bf ``disjoint'' version}] in the definition \eqref{eq:outer} of the outer measure, one may replace $c(E)$ by $c_\d(E)$. 
\end{description}
The latter property follows immediately by simple and well-known 

\begin{remark}\label{rem:CU}
For any sequence $(A_n)$ in $\A$ and the sequence $(B_n)$ defined by the condition $B_n=A_n\setminus 
\bigcup_{j<n}A_j$, one has the following: 
$(B_n)$ is a disjoint sequence in $\A$, $m(B_n)\le m(A_n)$ for all $n$, 
and 
$
\bigcup_n A_n=\bigcup_n B_n. 	
$
\end{remark}


Another useful property of the outer measure is just a bit more involved: 
\begin{proposition}\label{prop:1}
Take any disjoint sequence $(A_n)$ in $\A$. Then 
\begin{equation*}
	m^*\Big(\bigcup_n A_n\Big)=\sum_n m(A_n). 
\end{equation*}
\end{proposition}

\begin{proof}
Let $E:=\bigcup_n A_n$. Then trivially $(A_n)\in c_\d(E)\subseteq c(E)$, so that, by \eqref{eq:outer}, $m^*(E)\le\sum_n m(A_n)$. 

To prove the reverse inequality, take any $(B_k)\in c_\d(E)$ and any natural $N$. 
Let $C_N:=\cup_{n=1}^N A_n$, so that $C_N\in\A$. 
Then, by the $\si$-additivity of $m$ on $\A$, 
\begin{multline*}
	\sum_{n\le N} m(A_n)=\sum_{n\le N} \sum_k m(A_n\cap B_k)
	=\sum_k \sum_{n\le N} m(A_n\cap B_k)
	=\sum_k m(C_N\cap B_k) \\ 
	\le\sum_k m(B_k). 
\end{multline*}
Taking now the infimum over all $(B_k)\in c_\d(E)$ and recalling the ``disjoint'' version of the definition of the outer measure, we see that $\sum_{n\le N} m(A_n)\le m^*(E)$. 
Finally, letting $N\to\infty$, we confirm the reverse inequality, $\sum_n m(A_n)\le m^*(E)$, which completes the proof of Proposition~\ref{prop:1}. 
\end{proof}

An immediate and important consequence of Proposition~\ref{prop:1} is that 
\begin{equation}\label{eq:m*=m}
	m^*(A)=m(A) \quad\text{for all}\ A\in\A;  
\end{equation}
that is, $m^*$ is an extension of $m$ from $\A$ to the set of all subsets of $\Om$. 

\medskip

Note the following properties of the functions $d_A$: for any $A\in\A_\fin$ and any subsets $E,B,E_1,E_2,\dots,B_1,B_2,\dots$ of $\Om$, 
\begin{enumerate}[(I)]
	\item \label{i} $d_A$ is a pseudometric; 
	\item \label{ii} $d_A(E,B)=d_A(E^\c,B^\c)=d_A(E^\c,A\setminus B)$; 
	\item \label{iii} $d_A\Big(\bigcup\limits_n E_n,\bigcup\limits_n B_n\Big)\le\sum\limits_n
	 d_A(E_n,B_n)$ and \\ 
	$d_A\Big(\bigcap\limits_n E_n,\bigcap\limits_n B_n\Big)\le\sum\limits_n d_A(E_n,B_n)$;
	\item \label{iv} $m^*(A\cap E)\le m^*(A\cap B)+d_A(E,B)$. 
\end{enumerate}
Property \eqref{i} follows because the outer measure $m^*$ is nonnegative, monotone, and subadditive, whereas $E_1+E_3=E_1+E_2+E_2+E_3\subseteq(E_1+E_2)\cup(E_2+E_3)$ and $E_1+E_2=E_2+E_1$. 
Concerning Property \eqref{ii}, it is enough to note that $E+B=E^\c+B^\c$. 
To check Property \eqref{iii}, note that 
$\bigcup\limits_n E_n+\bigcup\limits_n B_n\subseteq\bigcup\limits_n(E_n+B_n)$ and 
$\bigcap\limits_n E_n+\bigcap\limits_n B_n\subseteq\bigcup\limits_n(E_n+B_n)$, and then use again the monotonicity and subadditivity of $m^*$. 
Finally, Property \eqref{iv} as well follows by the monotonicity and subadditivity of $m^*$, since 
$A\cap E\subseteq(A\cap B)\cup\big(A\cap(E+B)\big)$. 

Now we are ready to present

\begin{proof}[Proof of Theorem~\ref{th:si-alg}] 
Take any $A\in\A_\fin$ and any real $\vp>0$. 

Note first that $X\in\M$, since $d_A(X,A)=0<\vp$. 

Moreover, if $E\in\A$, then $d_A(E,B)=0<\vp$ for $B:=E\in\A$; so, 
$\M\supseteq\A$. 

That $\M$ is closed with respect to the complement easily follows from Property \eqref{ii} of $d_A$, which yields $d_A(E^c,A\setminus B)<\vp$ if $d_A(E,B)<\vp$. 

Also, Property \eqref{iii} of $d_A$ shows that $\M$ is closed with respect to the finite unions. So, $\M$ is an algebra.  

To complete the proof of Theorem~\ref{th:si-alg}, 
it remains to show that $\M$ is closed with respect to the countable unions. 

First here, take any disjoint sequence $(A_n)$ in $\A$. Then for any natural $N$ 
\begin{equation}\label{eq:cap A_n}
	d_A\Big(\bigcup_n A_n,\bigcup_{n\le N} A_n\Big)
	=m^*\Big(\bigcup\limits_{n>N} (A\cap A_n)\Big)=\sum_{n>N} m(A\cap A_n)
\end{equation}
by Proposition~\ref{prop:1}. 
On the other hand, for any natural $L$, 
\begin{equation*}
\sum_{n\le L} m(A\cap A_n)
=m\Big(\bigcup_{n\le L}(A\cap A_n)\Big)
\le m(A)<\infty, 	
\end{equation*}
since $A\in\A_\fin$. Hence, $\sum_n m(A\cap A_n)
\le m(A)<\infty$, which implies that \break 
$\sum_{n>N} m(A\cap A_n)\to0$ as $N\to\infty$. 
So, by \eqref{eq:cap A_n}, $\bigcup_n A_n\in\M$ -- for any disjoint sequence $(A_n)$ in $\A$. 

Moreover, since $\M$ is an algebra, in view of Remark~\ref{rem:CU} it now follows that $\bigcup_n B_n\in\M$ for any, not necessarily disjoint, sequence $(B_n)$ in $\A$. 

Finally, take any $E_1,E_2,\dots$ in $\M$. Then for each $n$ there is some $B_n\in\A$ such that $d_A(E_n,B_n)<\vp/2^n$. By the last paragraph, $\bigcup_n B_n\in\M$, and so, $d_A\big(\bigcup_n B_n,B)<\vp$ for some $B\in\A$. 
By Properties \eqref{i} and \eqref{iii} of $d_A$, 
\begin{equation*}
\begin{aligned}
d_A\Big(\bigcup\limits_n E_n,B\Big)
&\le
d_A\Big(\bigcup\limits_n E_n,\bigcup\limits_n B_n\Big) + d_A\Big(\bigcup\limits_n B_n,B\Big) \\ 
	&\le\sum\limits_n d_A(E_n,B_n)+\vp
	\le\sum\limits_n \vp/2^n+\vp=2\vp. 
\end{aligned}	 
\end{equation*}
This completes the proof of Theorem~\ref{th:si-alg}. 
\end{proof}

\begin{proof}[Proof of Theorem~\ref{th:si-add}]  
In view of the subadditivity property of $m^*$, it is enough to show that $m^*$ is finitely superadditive; that is, for any disjoint $E_1$ and $E_2$ in $\M$, one has 
\begin{equation}\label{eq:super}
	m^*(E_1\cup E_2)\overset{\text{?}}\ge m^*(E_1)+m^*(E_2). 
\end{equation}

Take indeed any such $E_1$ and $E_2$. 
Take also any real $\vp>0$. 
If $m^*(E_1\cup E_2)=\infty$, then inequality \eqref{eq:super} is trivial. So, without loss of generality $m^*(E_1\cup E_2)<\infty$. 
Hence, in view of the ``disjoint'' version of the definition of the outer measure, for some sequence $(A_n)\in c_\d(E_1\cup E_2)$ we have 
\begin{equation}\label{eq:sum m<infty}
\sum_n m(A_n)\le m^*(E_1\cup E_2)+\vp<\infty, 	
\end{equation}
and so, for any natural $N$ and $C:=C_N:=\bigcup_{n\le N}A_n\in\A_\fin$, 
\begin{equation}\label{eq:infty>}
	\infty>m^*(E_1\cup E_2)\ge\sum_n m(A_n)-\vp\ge\sum_{n\le N} m(A_n)-\vp=m(C)-\vp. 
\end{equation}
Further, since $E_1$ and $E_2$ are in $\M$, one can find $B_1$ and $B_2$ in $\A$ such that 
\begin{equation}\label{eq:d<ep}
	d_C(E_\al,B_\al)<\vp; 
\end{equation}
here and in what follows, $\al$ is $1$ or $2$. 
By Property~\ref{iv} of the pseudometrics $d_A$, 
\begin{equation}\label{eq:m*1}
	m^*(C\cap E_\al)\le m^*(C\cap B_\al)+d_C(E_\al,B_\al)\le m^*(C\cap B_\al)+\vp. 
\end{equation}
On the other hand, using first the monotonicity 
of $m^*$ and the condition $(A_n)\in c_\d(E_1\cup E_2)$, and then Proposition~\ref{prop:1} and  \eqref{eq:sum m<infty}, 
we have
\begin{equation}\label{eq:m*2}
	m^*(C^\c\cap E_\al) 
	\le m^*\big(C^\c\cap\bigcup_n A_n\big)
	=m^*\big(\bigcup_{n>N} A_n\big) 
	=\sum_{n>N}m(A_n)<\vp 
\end{equation}
if $N$ is large enough -- which latter will be assumed in the sequel. 
It follows from the subadditivity of $m^*$, \eqref{eq:m*1}, \eqref{eq:m*2}, and \eqref{eq:m*=m} that 
\begin{equation*}
	m^*(E_\al)\le m^*(C\cap E_\al)+m^*(C^c\cap E_\al)
	\le m^*(C\cap B_\al)+2\vp=m(C\cap B_\al)+2\vp. 
\end{equation*}
Therefore, 
\begin{equation}\label{eq:m*+m*<}
\begin{aligned}
	m^*(E_1)+m^*(E_2)-4\vp&\le m(C\cap B_1)+m(C\cap B_2) \\ 
	&= m\big(C\cap (B_1\cup B_2)\big)+m(C\cap B_1\cap B_2) \\ 
	&\le m(C)+d_C(E_1,B_1)+d_C(E_2,B_2)\le m(C)+2\vp; 
\end{aligned}
\end{equation}
the penultimate inequality here holds by Property~\eqref{iii} of the pseudometrics $d_A$, taking also into account that $C\cap (B_1\cup B_2)\subseteq C$ and $C\cap E_1\cap E_2\subseteq E_1\cap E_2=\emptyset$, whereas the last inequality in \eqref{eq:m*+m*<} follows immediately from \eqref{eq:d<ep}. 
Comparing the multi-line display \eqref{eq:m*+m*<} with \eqref{eq:infty>}, we see that 
$m^*(E_1\cup E_2)\ge m^*(E_1)+m^*(E_2)-7\vp$, which concludes the proof of \eqref{eq:super} and thus the proof of Theorem~\ref{th:si-add}.   
\end{proof}

\begin{proof}[Proof of Theorem~\ref{th:uniq}]
Recall that the $\si$-finiteness of the measure $m$ on $\A$ means that there is a disjoint sequence $(D_n)$ in $\A_\fin$ such that  
$\bigcup_n D_n=\Om$. In the rest of this proof, let $(D_n)$ be such a sequence. 

In addition to 
the restriction $\bar m$ of $m^*$ to $\M$, let $\tm$ be another measure that extends $m$ from the algebra $\A$ to the $\si$-algebra $\M$. 
Take any $E\in\M$. 

By the ``disjoint'' version of the definition of the outer measure, for each real $r>m^*(E)$ there is a sequence $(A_n)\in c_\d(E)$ such that $\sum_n m(A_n)<r$. So,
$\tm(E)\le\tm\big(\bigcup_n A_n\big)=\sum_n\tm(A_n)=\sum_n m(A_n)<r$, for any real $r>m^*(E)$. 
Thus, $\tm(E)\le m^*(E)=\bar m(E)$. 

So, for each $n$ one has $\tm(D_n\cap E)\le\bar m(D_n\cap E)\le m(D_n)<\infty$ and \break 
$\tm(D_n\cap E^\c)\le\bar m(D_n\cap E^\c)\le m(D_n)<\infty$. Adding now inequalities \break 
$\tm(D_n\cap E)\le\bar m(D_n\cap E)$ and $\tm(D_n\cap E^\c)\le\bar m(D_n\cap E^\c)$, we get $\tm(D_n)<\bar m(D_n)$ unless $\tm(D_n\cap E)=\bar m(D_n\cap E)$. But $\tm(D_n)<\bar m(D_n)$ contradicts the condition that both $\bar m$ and $\tm$ are extensions of the measure $m$ on $\A$. 

Thus, $\tm(D_n\cap E)=\bar m(D_n\cap E)$ for all $n$, whence 
$\tm(E)=\sum _n\tm(D_n\cap E)=\sum _n\bar m(D_n\cap E)=\bar m(E)$, for all $E\in\M$.  
\end{proof}

The following characterization of the outer measure will be useful in the proof of Theorem~\ref{th:M comp}, and it may also be of independent interest. 

\begin{proposition}\label{prop:tight}
Take any $E\subseteq\Om$. Then 
\begin{equation*}
\begin{aligned}
	m^*(E)&=m^*_\circ(E):=\inf\{\bar m(S)\colon S\in\M,\ S\supseteq E\} \\ 
	&=m^*_{\circ\circ}(E):=\inf\{\bar m(S)\colon S\in\si(\A),\ S\supseteq E\}. 
\end{aligned}	
\end{equation*}
Morever, the second of the two infima is attained, and hence the first one is attained. 
\end{proposition}

\begin{proof}
That $m^*_{\circ}(E)\le m^*_{\circ\circ}(E)$ follows because $\si(\A)\subseteq\M$. That $m^*(E)\le m^*_{\circ}(E)$ follows because for any $S\in\M$ such that $S\supseteq E$ one has $m^*(E)\le m^*(S)=\bar m(S)$. So, $m^*(E)\le m^*_{\circ}(E)\le m^*_{\circ\circ}(E)$. 

Thus, to complete the proof of Proposition~\ref{prop:tight}, 
it is enough to construct some $S\in\si(\A)$ such that $S\supseteq E$ and $\bar m(S)\le m^*(E)$. Such a construction is easy. Indeed, for each natural $k$ there is a sequence $\big(A_n^{(k)}\big)_{n=1}^\infty\in c_\d(E)$ such that $\si(\A)\ni B^{(k)}:=\bigcup_n A_n^{(k)}\supseteq E$ and 
$\bar m(B^{(k)})=\sum_n m(A_n^{(k)})\le m^*(E)+1/k$. 
Let now $S:=\bigcap_k B^{(k)}$. Then indeed $S\in\si(\A)$, $S\supseteq E$, and $\bar m(S)\le m^*(E)$. 
\end{proof}

\begin{proof}[Proof of Theorem~\ref{th:M comp}] The first sentence of Theorem~\ref{th:M comp} will be verified in the following six steps. 

\textbf{Step 1: verification of $\M_\Co\subseteq\M_{\Co,\A_\fin}$.} 
This follows immediately, because $0\le d_A(E,S)\le d_X(E,S)$ for any subsets $E$ and $A$ of $X$. 

\textbf{Step 2: verification of $\M_{\Co,\A_\fin}\subseteq\M$.} 
Take any $E\in\M_{\Co,\A_\fin}$. Take next any real $\vp>0$ and any $A\in\A_\fin$. Then $d_A(E,S)=0$ for some $S\in\si(\A)$. Take any such $S$. Then, by Theorem~\ref{th:si-alg}, $S\in\M$ and therefore $d_A(S,B)<\vp$ for some $B\in\A$, whence, by the triangle inequality, $d_A(E,B)\le d_A(E,S)+d_A(S,B)<\vp$. So, $E\in\M$. 
Step 2 is complete. 

\textbf{Step 3: verification of $\M_{\Co,\A_\fin}\supseteq\M$.} 
Take any $E\in\M$. Take next any real $\vp>0$ and any $A\in\A_\fin$. Then for each natural $k$ there is some set $B_k\in\A$ such that $d_A(E,B_k)\le\vp/2^k$. 
Let $C:=\bigcup_j C_j$, where $C_j:=\bigcap_{k>j}B_k$. 
Note that $C\in\si(\A)$ and $C=\bigcup_{j>N} C_j$ for any natural $N$, since $C_j\subseteq C_{j+1}$ for all $j$. So, by Property~\ref{iii} of $d_A$, we have 
$d_A(E,C_j)\le\sum_{k>j}d_A(E,B_k)\le\sum_{k>j}\vp/2^k=\vp/2^j$ for all $j$, whence 
$d_A(E,C)\le\sum_{j>N}d_A(E,C_j)\le\vp/2^N$, for any natural $N$. 
Thus, $d_A(E,C)=0$, which means that $E\in\M_{\Co,\A_\fin}$.  
Step 3 is complete. 


\textbf{Step 4: verification of $\M\subseteq\M_\Ca$.} 
Take any $E\in\M$ and then any $F\subseteq\Om$. By Proposition~\ref{prop:tight}, for some $S\in\M$ one has $S\supseteq F$ and $\bar m(S)=m^*(F)$. Then 
$S\cap E\in\M$, $S\cap E^\c\in\M$, $S\cap E\supseteq F\cap E$, and $S\cap E^\c\supseteq F\cap E^\c$, whence 
\begin{equation*}
\begin{aligned}
	m^*(F)=\bar m(S)=\bar m(S\cap E)+\bar m(S\cap E^\c)
	&=m^*(S\cap E)+m^*(S\cap E^\c) \\ 
	&\ge m^*(F\cap E)+m^*(F\cap E^\c), 
\end{aligned}	
\end{equation*}
so that $m^*(F)\ge m^*(F\cap E)+m^*(F\cap E^\c)$. The reverse inequality follows by 
the subadditivity of $m^*$. Thus, $E\in\M_\Ca$, for any $E\in\M$. Step 4 is complete. 

\textbf{Step 5: verification of $\M_\Ca\subseteq\M_{\Ca,\A_\fin}$.} This is trivial. 

\textbf{Step 6: verification of $\M_{\Ca,\A_\fin}\subseteq\M$.}  
Take any $E\in\M_{\Ca,\A_\fin}$ and then any $A\in\A_\fin$ and any real $\vp>0$. We want to show here that $d_A(E,B)\le3\vp$ for some $B\in\A$.  
The conditions $A\in\A_\fin$ and $E\in\M_{\Ca,\A_\fin}$ yield 
\begin{equation}\label{eq:Ca}
	\infty>m(A)=m^*(A)=m^*(A\cap E)+m^*(A\cap E^\c). 
\end{equation}
Next, for some sequences $(S_n)\in c_\d(A\cap E)$ and $(T_n)\in c_\d(A\cap E^\c)$ we have 
\begin{equation}\label{eq:subset M}
\begin{alignedat}{2}
	&\M\ni S:=\bigcup_n S_n\supseteq A\cap E,\quad  &&\bar m(S)=\sum_n m(S_n)\le m^*(A\cap E)+\vp, \\ 
	&\M\ni T:=\bigcup_n T_n\supseteq A\cap E^\c,\quad  &&\bar m(T)=\sum_n m(T_n)\le m^*(A\cap E^\c)+\vp.  
\end{alignedat}
\end{equation}
Moreover, without loss of generality $S\cup T\subseteq A$; otherwise, replace $S_n$ and $T_n$ by $A\cap S_n$ and $A\cap T_n$, respectively. 
Since $S\supseteq A\cap E$ and $T\supseteq A\cap E^\c$, 
it follows that 
\begin{equation}\label{eq:S,T,A}
	S\cup T=A. 
\end{equation}
By the first inequality in \eqref{eq:subset M}, $\sum_n m(S_n)<\infty$,  
because 
$m^*(A\cap E)\le m^*(A)=m(A)<\infty$. So, for some natural $N$ we have $\sum_{n>N} m(S_n)\le\vp$. 
Let now 
$
	B:=\bigcup_{n\le N}S_n. 
$ 
Then $B\in\A$ and 
\begin{equation}\label{eq:d(S,B)}
	d_A(S,B)=m^*\Big(\bigcup_{n>N}S_n\Big)=\sum_{n>N} m(S_n)\le\vp. 
\end{equation}
Since $S\subseteq A$ and $A\cap E^\c\subseteq T$, we have $S\setminus(A\cap E)
\subseteq S\cap T$. Therefore and in view of \eqref{eq:subset M}, \eqref{eq:S,T,A}, and \eqref{eq:Ca}, 
\begin{multline*}
	d_A(E,S)=m^*\big(S\setminus(A\cap E)\big)
	\le m^*(S\cap T)=\bar m(S\cap T) \\
	=\bar m(S)+\bar m(T)-\bar m(A)  
	\le m^*(A\cap E)+\vp+m^*(A\cap E^\c)+\vp-m(A)=2\vp. 
\end{multline*}
This and \eqref{eq:d(S,B)} yield the desired result: 
\begin{equation*}
	d_A(E,B)\le d_A(E,S)+d_A(S,B)\le3\vp.   
\end{equation*}
This completes Step~6 and thus the entire proof of the first sentence of Theorem~\ref{th:M comp}. 


It remains to verify its second sentence. To do this, assume that $m$ is $\si$-finite, so that there is a disjoint sequence $(D_n)$ in $\A_\fin$ such that $\bigcup_n D_n=\Om$. Take now 
any $E\in\M_{\Co,\A_\fin}$. 
Since $E\in\M_{\Co,\A_\fin}$, for each natural $n$ there is some set $S_n\in\si(\A)$ such that $d_{D_n}(E,S_n)=0$. Here one can replace $S_n$ by $D_n\cap S_n$, so that without loss of generality $S_n\subseteq D_n$. Let then $S:=\bigcup_n S_n$, so that $S\in\si(\A)$ and $D_n\cap S=S_n$ for each $n$. Since $m^*$ is $\si$-additive on $\M=\M_{\Co,\A_\fin}$ and $\M\supseteq\si(\A)$, 
it follows that 
$d_X(E,S)=\sum_n d_{D_n}(E,S)=\sum_n d_{D_n}(E,D_n\cap S)=\sum_n d_{D_n}(E,S_n)=0$. 
So, $E\in\M_\Co$, for any $E\in\M_{\Co,\A_\fin}$ -- if $m$ is $\si$-finite.


Theorem~\ref{th:M comp} is now completely proved. 
\end{proof}





\bibliographystyle{abbrv}

\bibliography{P:/mtu_pCloud_02-02-17/bib_files/citations12.13.12}


\end{document}